\documentclass[12pt]{article}

\usepackage{graphicx}  
\usepackage{amsmath} 
\usepackage{url}
\usepackage{amsfonts}  
\usepackage{amsthm}  

\newtheorem{thm}{Theorem}[section]
\newtheorem{lem}[thm]{Lemma}
\newtheorem{prop}[thm]{Proposition}
\newtheorem{cor}[thm]{Corollary}
\newtheorem{conj}[thm]{Conjecture}

\def\p{\mathbb P}
\def\e{\mathbb E}
\def\lr{\left(}
\def\rr{\right)}
\usepackage{color}

\usepackage{blindtext}
\title{The Number of Seymour Vertices in Random Tournaments and Digraphs}
\author{Zachary Cohn\\University of Chicago\\{\tt cohn.zachary@gmail.com}\and Anant Godbole\\ East Tennessee State University\\{\tt godbolea@etsu.edu}\and Elizabeth Wright Harkness\\ Tulane University\\{\tt liz.wright.harkness@gmail.com}\and Yiguang Zhang\\The Johns Hopkins University\\{\tt yzhan132@jhu.edu}}
\begin{document}
\maketitle
\begin{abstract}
Seymour's distance two conjecture states that in any digraph there exists a vertex (a ``Seymour vertex") that has at least as many neighbors at distance two as it does at distance one. We explore the validity of probabilistic statements along lines suggested by Seymour's conjecture, proving that almost surely there are a ``large" number of Seymour vertices in random tournaments and ``even more" in general random digraphs.
\end{abstract}
\section{Introduction}
\subsection{Notation}
For the purpose of this paper, a \emph{digraph} exclusively means a simple, directed graph without loops or multiple edges (including edges in the same direction and antiparallel edges).\\
\indent For any pair of vertices $u,v$ in a digraph $D$, the length of the shortest directed path from $u$ to $v$ in $D$ is denoted as $dist(u,v)$. We write $N_i(u)$ to denote the set of vertices that are at distance $i$ from $u$. A vertex $v_0 \in V(D)$ is called a \emph{Seymour vertex} if $|N_2(v_0)|\geq |N_1(v_0)|$. We write $S$ for the set of all Seymour vertices in the digraph.
\subsection{Background: Seymour's Conjecture}
\begin{conj}(Seymour's Second Neighborhood Conjecture). If $D$ is a directed graph with no loops or multiple edges, then $D$ has a vertex $v_0$ such that $|N_2(v_0)| \geq |N_1(v_0)|$.
\end{conj}
Although the proof of this conjecture remains open, several partial results have been proved over the last two decades:  \begin{thm}(Kaneko and Locke \cite{kaneko})  Seymour's conjecture is true if the minimum outdegree of vertices in $D$ is at most 6.
\end{thm}

\bigskip

\noindent {\bf Dean's Conjecture}  {\it Seymour's conjecture is true if $D$ is any tournament $T$.}

\bigskip

\begin{thm} (Fisher \cite{fisher})  Dean's conjecture is true.
\end{thm}

\bigskip

Chen, Shen, and Yuster \cite {chen} have shown that for every digraph, there is a vertex $v$ such that $\vert N_2(v)\vert \ge r\vert N_1(v)\vert$, where $r\approx .657$, and they state a further improvement to $r\approx .678$. See the website \cite{west2} for details.
Seymour's conjecture may be seen as a special case of a more general 1988 conjecture of Caccetta and H\"aggkvist:

\bigskip

\noindent {\bf Caccetta-H\"aggkvist Conjecture \cite{caccetta}} {\it If $D$ is a simple  digraph on $n$ vertices, and each vertex has outdegree at least $d$, then the girth of $D$ (the length of the shortest directed cycle) is at most $n/d$.}

\bigskip

\noindent The Caccetta-H\"aggkvist conjecture been proved for $d=2,3,4,5,\frac{n}{2}$.  See Douglas West's website \cite{west3}
for several related results pertinent to the conjecture.  The truth of Seymour's conjecture would settle the important ``balanced" $d=\frac{n}{3}$ case of the Caccetta-H\"aggkvist conjecture, i.e., when each vertex has in- {\it and} out degree at least $d=\frac{n}{3}$.  A short proof of this fact follows in the case $3\vert n$:

We need to prove that $D$ has a directed triangle.  We let $v$ be a Seymour vertex, and note that the other vertices separate themselves out into 
$N_1(v), \vert N_1(v)\vert\ge n/3;$ 
$N_{-1}(v), \vert N_{-1}(v)\vert\ge n/3,$ 
(where $N_{-1}(v)$ consists of those vertices that point towards $v$); and
$N_0(v), \vert N_0(v)\vert< n/3,$ 
which are those vertices that have no edge to or from $v$.  Now if there is an edge from $u\in N_1(v)$ to $w\in N_{-1}(v)$, then $v,u,$ and $w$ create a directed triangle, and we are done.  On the other hand, if there is no such edge, then vertices in $N^-(v)$ cannot be at distance two, forcing all distance two vertices to be in $N_0(v)$, which leads to the contradiction that
$$\vert N_2(v)\vert\le\vert N_0(v)\vert<n/3\le\vert N_1(v)\vert.$$
\hfill\qed

In this paper\footnote{This work was started by the first three authors, reported on at \cite{west2}, and completed this year by Godbole and Zhang.}, we study the {\it number} $S=S_n=S_{n,p}$ of Seymour vertices in {random tournaments} and random digraphs.  Actually, our proofs will reveal that Nate Dean's conjecture, proved by Fisher in \cite{fisher}, is very insightful: in particular, we will see that there are many more Seymour vertices in random digraphs with $p<1/2$ (definitions below) than there are in random tournaments, and the tightness of the concentration is greater in the former case.  

Specifically, it is shown in Section 2 that there are close to $\frac{n}{2}$ Seymour vertices in random tournaments with high probability, where ``close to" and ``with high probability" are interpreted in a variety of ways.  In particular, both convergence in measure and almost everywhere convergence are invoked.  An interesting variance computation in this section shows that there is an oscillation in the number of Seymour vertices as we add additional vertices to the tournament, and this reflects itself in the piecewise linear ``even-odd" dichotomy in the variance of the number of Seymour vertices.  After methods such as the exponential inequalities of Azuma and Talagrand failed, we used skeletal subsequences of polynomial size (along with an analysis of maximal deviation between these checkpoints) to establish inequalities that yield the almost everywhere convergence referred to above.

  In Section 3, we consider random digraphs on $n$ vertices, and show that the probability that every vertex is a Seymour vertex tends to 1 as $n\to\infty$, provided that the edge probability $p$ satisfies $o(1)<p<\frac{1}{2}-o^*(1)$ for well-specified $o(1)$ and $o^*(1)$ functions.

\section{Random Tournaments}
An \emph{orientation} of graph $G$ is a digraph $D$ obtained from $G$ by choosing an orientation ($u\rightarrow v$ or $v \rightarrow u$) for each edge $uv \in E(G)$. A \emph{tournament} is an orientation of a complete graph $K_n$. Our model for a random tournament $T_n$ is the probability space of all possible orientations of the complete graph $K_n$, chosen in an equiprobable fashion. Equivalently, the orientation of each edge $uv \in E(K_n)$  is chosen independently as $u \rightarrow v$ or $v \rightarrow u$ with probability $1/2$.
\begin{prop}
Let $T_n$ be a random tournament and $S$ the set of its Seymour vertices.  Then as $n\to\infty$
\[ \mathbb{E}(|S|) \sim \frac{n}{2}(1+o(1))\]
as $n\to\infty$.
\end{prop}
\begin{proof}
Let $X:=\vert S\vert$, and for $i \in [n]$, define
\[X_i = \left\{ \begin{array}{ll}1\ & \text{vertex }\mbox{$i$}\text{ is a Seymour vertex}\\0 & \text{otherwise}\end{array}\right.\:\:\: \text{so that}\ X = \sum_{i=1}^{n}X_i.\]
By linearity of expectation, 
\begin{eqnarray}\mathbb{E}(X)&=&n \mathbb{P}( 1 \in S)\nonumber\\&=& n\mathbb{P}(1\in S; |N_1(1)|+|N_2(1)|=n-1)\nonumber\\{}&&+ n\mathbb{P}(1\in S; |N_1(1)|+|N_2(1)|<n-1)\nonumber\\
&\leq& n\mathbb{P}(1\in S; |N_1(1)|+|N_2(1)|=n-1)\nonumber\\{}&&+ n\mathbb{P}(|N_1(1)|+|N_2(1)|<n-1),
\end{eqnarray}
and 
\begin{eqnarray}
\mathbb{E}(X) &\geq& n\mathbb{P}(1\in S; |N_1(1)|+|N_2(1)|=n-1)\nonumber\\ &=& n\mathbb{P}(|N_1(1)|\leq (n-1)/2 ; |N_1(1)|+|N_2(1)|=n-1).
\end{eqnarray}
By (1),
\begin{eqnarray*}\e(X)&\leq& n\mathbb{P}(1\in S;  |N_1(1)|+|N_2(1)|=n-1)+ n(n-1)\mathbb{P}(dist(1,2) \geq 3)\\
&\leq& n\p\lr\vert N_1(1)\vert\le\frac{n-1}{2}\rr+ \frac{n(n-1)}{2}\lr\frac{3}{4}\rr^{n-2},
\end{eqnarray*}
since to have $dist(1,2)\geq 3$, the edge $1\rightarrow 2$ must be absent. Furthermore, for any vertex $i \in \{3,4 \cdots n\}$, $1\rightarrow i$ and $i \rightarrow 2$ cannot both be present.

\indent The second term is exponentially small and, in  the first term, $\mathbb{P}(|N_1(1)|\leq \frac{n-1}{2})$ is clearly $\frac{1}{2}$ if $n$ is even; if $n$ is odd, then $\mathbb{P}(|N_1(1)|\leq \frac{n-1}{2}) = \frac{1}{2} + \frac{1}{2}\mathbb{P}(|N_1(1)| = \frac{n-1}{2}) = \frac{1}{2}+\frac{1}{2}\frac{(n-1)!}{((n-1)/2)!^2}\frac{1}{2^{n-1}}.$ A Stirling approximation gives $\mathbb{P}(|N_1(1)| = \frac{n-1}{2}) \sim \sqrt{\frac{2}{{\pi (n-1)}}}$, so that
\[\mathbb{E}(X) \leq \left\{ \begin{array}{ll}\mbox{$\frac{n}{2}(1+o_1(1))$} & \text{ if $n$ is even} \\\mbox{$\frac{n}{2}(1+o_2(1))$} & \text{ if $n$ is odd,}  \end{array}\right.\]
where $o_1(1)$ is exponentially small and $o_2(1)=O(1/\sqrt{n})$.  Since $\mathbb{P}( |N_1(1)|+|N_2(1)|=n-1) = 1- \mathbb{P}(|N_1(1)|+|N_2(1)|<n-1) \ge 1- (n-1)\frac{1}{2}(\frac{3}{4})^{n-2}$, (2) gives
$$\e(X)\geq n\mathbb{P}\big(|N_1(1)|\leq (n-1)/2\big) -n(n-1)\frac{1}{2}\lr\frac{3}{4}\rr^{n-2}.$$
A similar analysis as above now establishes the result.  
\end{proof}
The difference in the $o(1)$ functions in the above result proves to be highly significant -- one of its immediate ramifications, seen in the next proposition, is that the variance of the number of Seymour vertices grows linearly, but in a piecewise fashion.  Other less obvious complications might indeed be caused by this ``difference in the even and odd cases."
\begin{prop}
Let $T_n$ be a random tournament and $S$ be the set of its Seymour vertices.  Then for constants $C_1$ and $C_2$,
$Var(|S|) \sim C_1n(1+o(1))$ as $n\to\infty$ if $n$ is even, and $Var(|S|) \sim C_2n(1+o(1))$ as $n\to\infty$ if $n$ is odd.
\end{prop}
\begin{proof} Since
$$Var(X)= \sum_{i=1}^n[\mathbb{E}(X_i) - \mathbb{E}^2(X_i)] + 2 \sum_{i<j}[\mathbb{E}(X_iX_j) - \mathbb{E}(X_i)\mathbb{E}(X_j)],$$
and $\mathbb{E}(X_iX_j) =\mathbb{E}(X_1X_2)= \mathbb{P}(1,2\in S)=2\mathbb{P}(1,2\in S;  1\rightarrow 2),$
the key term in the above display for the variance is given by $\p(1,2\in S;  1\rightarrow 2)=p_1+p_2+p_3+p_4$, where
$$p_1=\mathbb{P}(1,2\in S; 1\rightarrow 2; A_1\cap A_2),$$
$$p_2=\mathbb{P}(1,2\in S; 1\rightarrow 2 ;  A_1^C\cap A_2),$$
$$p_3=\mathbb{P}(1,2\in S; 1\rightarrow 2; A_1\cap A_2^C),$$
and
$$p_4=\mathbb{P}(1,2\in S; 1\rightarrow 2; A_1^C\cap A_2^C),$$ 
and  $A_i$, $i=1,2$, are the events that all vertices are at distance no more than 2 from vertex $i$.
Since $p_2,p_3, p_4\le\p(A_1^C)\le \frac{n}{2}(3/4)^{n-2}$, $p_1$ is the dominant term.
\begin{eqnarray*}
 p_1&\le&\mathbb{P}\bigg(1\rightarrow 2; |N_1(1)\backslash \{2\}|\leq \frac{n-1}{2}-1; |N_2(2)\backslash\{1\}|\geq \frac{n-1}{2}-1\bigg)\\
&=& \frac{1}{2}\mathbb{P}\lr|N_1(1)\backslash \{2\}|\leq \frac{n-1}{2}-1 \rr\times \mathbb{P}\lr|N^*_1(2)|\leq \frac{n-1}{2}\rr\\
&=&{\frac{1}{2}{\sum_{k=0}^{\lfloor \frac{n-1}{2}-1\rfloor }{{n-2}\choose k} \lr\frac{1}{2}\rr^k\lr\frac{1}{2}\rr^{n-2-k}}\times {\sum_{k=0}^{\lfloor \frac{n-1}{2}\rfloor }{{n-2}\choose k} \lr\frac{1}{2}\rr^k\lr\frac{1}{2}\rr^{n-2-k}}},
\end{eqnarray*}
where $N^*_1(2)$ is the first neighborhood of vertex 2 in the set $\{3,4,\ldots,n\}$.  
Notice that if $n$ is even, the first term above equals $\frac{1}{2}-\frac{1}{2}\p(\text{Bin}(n-2,0.5)=\frac{n-2}{2})$ and the second equals $\frac{1}{2}+\frac{1}{2}\p(\text{Bin}(n-2,0.5)=\frac{n-2}{2})$.  If $n$ is odd, however, the first factor is exactly 1/2, while the second equals $\frac{1}{2}+\p(\text{Bin}(n-2,0.5)=\frac{n-1}{2})$.  

Let $n$ be even.  Then, considering the proof of Proposition 2.1 and denoting by $o_1(1)$ a generic function that decays exponentially, we have that
\begin{eqnarray}
Var(X)&=& \sum_{i=1}^n[\mathbb{E}(X_i) - \mathbb{E}^2(X_i)] + 2 \sum_{i<j}[\mathbb{E}(X_iX_j) - \mathbb{E}(X_i)\mathbb{E}(X_j)]\nonumber\\
&\le&n\lr
\frac{1}{2}-\frac{1}{4}+o_1(1)\rr\nonumber\\
{}&&+n^2\bigg(\frac{1}{4}-\frac{1}{4}\p^2\lr\text{Bin}(n-2,0.5)=\frac{n-2}{2}\rr\nonumber\\
{}&&\qquad+2p_2+2p_3+2p_4-\lr\frac{1}{4}+o_1(1)\rr\bigg)\nonumber\\
&=&n\lr\frac{1}{4}\rr-\frac{n^2}{2\pi n}(1+o(1))+o_1(1)\nonumber\\
&=&n\lr\frac{1}{4}-\frac{1}{2\pi}\rr(1+o(1)),
\end{eqnarray}
since $\p(\text{Bin}(n-2,0.5)=\frac{n-2}{2})\sim{\sqrt{\frac{2}{\pi n}}}$ by Stirling's approximation.

If $n$ is odd, we have, on the other hand,
\begin{eqnarray}
Var(X)&=& \sum_{i=1}^n[\mathbb{E}(X_i) - \mathbb{E}^2(X_i)] + 2 \sum_{i<j}[\mathbb{E}(X_iX_j) - \mathbb{E}(X_i)\mathbb{E}(X_j)]\nonumber\\
&\le&n\lr
\frac{1}{4}-\frac{1}{4}\p^2\lr\text{Bin}(n-1,0.5)=\frac{n-1}{2}\rr\rr\nonumber\\
{}&&+n^2\bigg(\frac{1}{4}+\frac{1}{2}\p\lr\text{Bin}(n-2,0.5)=\frac{n-1}{2}\rr\nonumber\\
{}&&\qquad+2p_2+2p_3+2p_4-\lr\frac{1}{2}+\frac{1}{2}\p\lr\text{Bin}(n-1,0.5)=\frac{n-1}{2}\rr\rr^2\bigg)\nonumber\\
&=&\frac{n}{4}(1+o(1))+\frac{n^2}{2}(\pi_2-\pi_1)-\frac{n^2}{4}\pi_1^2+o_1(1),
\end{eqnarray}
where \[\pi_1=\p\lr\text{Bin}(n-1,0.5)=\frac{n-1}{2}\rr\] and \[\pi_2=\p\lr\text{Bin}(n-2,0.5)=\frac{n-1}{2}\rr.\]  Since $$\pi_2-\pi_1\sim{\sqrt{\frac{2}{\pi}}}\lr\frac{1}{{\sqrt{n-2}}}-\frac{1}{{\sqrt{n-1}}}\rr(1+o(1))={\sqrt{\frac{2}{\pi}}}\frac{1}{n}(1+o(1)),$$
and
\[\pi_1^2\sim\frac{2}{\pi n},\]
it follows from (4) that
\begin{equation}Var(X)\le n\lr\frac{1}{4}-\frac{1}{2\pi}+\frac{1}{\sqrt{2\pi}}\rr(1+o(1))\end{equation}
in the odd case.  It is now straightforward to get matching lower bounds of the same order of magnitude as in (3) and (5).  This proves the result.\end{proof}
A natural question to ask is {\it why} there isn't a uniform growth rate for the variance.  Here is a heuristic reason:  Even though the expected values in both the even and odd cases are $\sim n/2$, the second order terms are significant.  Suppose we have observed the tournament with an even number of vertices.  By Stirling's formula, about $C \sqrt n$ of the vertices $v$ are ``borderline Seymour," meaning that $i(v)-o(v)=1$ and about $C\sqrt n$ are borderline non-Seymour, i.e., satisfy $o(v)-i(v)=1$ -- where $i(\cdot)$ and $o(\cdot)$ are the in- and out-degree functions.  When a new vertex ``joins" the tournament, notice that we cannot lose Seymour vertices, but borderline non-Seymour vertices have a 0.5 chance of becoming borderline Seymour, with $i(v)=o(v)$.  There is thus an increase in the $\e(\vert S\vert)$ by $\sim (C/2)\sqrt n$, an increase that almost gets nullified when a second new vertex joins the tournament ($n$ becomes even again) and borderline Seymour vertices become borderline non-Seymour. This dynamic evolution of the number of Seymour vertices causes an ebb and flow in the variance also, as reflected by Proposition 2.2.

\begin{prop}As $n$ goes to infinity,
$$\mathbb{P}\lr\left\vert |S|-\mathbb{E}(|S|)\right\vert \ge A\sqrt{{n \log n}}\rr\rightarrow 0.$$
\end{prop}
\begin{proof}
Immediate from Chebychev's inequality and Propositions 2.1 and 2.2, which indicate that for any $A>0$,
\begin{equation}\mathbb{P}\lr\left\vert |S|-\mathbb{E}(|S|)\right\vert \ge A\sqrt{{n \log n}}\rr\le\frac{K}{A^2\log n}\end{equation} for some constant $K$.
\end{proof}
\noindent {\bf Discussion.}  The above rate of convergence is unsatisfactory; for reasons to be made clearer, we would like to have a summable upper bound on the probability in (6).  This is equivalent to finding an exponential inequality to bound the probability.  Accordingly, we first attempted to use Azuma's inequality as found in \cite{alon}.  Here it turns out that a change in the orientation of a single edge can, in the worst case scenario, change the value of $S$ quite dramatically.  However, it can be shown that if the tournament is of diameter 2, and if a change in any edge orientation does not change the diameter, then $S$ cannot change by more than 2 and we have a 2-Lipschitz situation.  The probability of this, moreover, can be shown to be $1-\varepsilon_n$, where $\varepsilon_n$ is exponentially small. A modified version of Azuma's inequality, in which such small exceptional probabilities are allowed, may be found in \cite{chung}, Theorem 2.37 -- but this too proves to give us a width of concentration of $\Omega(n)$ around $\e(S)\sim n/2$, since there are a quadratic number of edges in the ``edge exposure martingale."  Using the vertex exposure martingale vastly changes the maximal change in $S$, and thus provides no improvement.  Likewise, Talagrand's inequality \cite{alon} involves a very large linear certification function, and is similarly unable to squeeze out a better upper bound in (3).  We thus resort to ``Chebychev's inequality on blocks" to prove the next result.  (Azuma's inequality on blocks, as methodically exploited by Frieze (see, e.g., \cite{steele}), could conceivably be used also.)
\begin{thm}  For each $\epsilon>0$,
\[\p\lr\vert S_n-\e(S_n)\vert>n^{0.5+\epsilon}\ {\rm infinitely\ often}\rr=0.\]
\end{thm}
\begin{proof}  Let us prove equivalently that for each $\epsilon>0$, 
\[\p\lr\vert S_n-\e(S_n)\vert>n^{0.5+\epsilon}\log n\ \text{infinitely often}\rr=0,\] illustrating the method for $\epsilon=1/4$.  We have, by Chebychev's inequality,
\[\p\lr\vert S_{n^2}-\e(S_{n^2})\vert>n^{3/2}\log n\rr\le\frac{K}{n\log^2 n}\]  for some $K>0$.  Since the right side is summable, we use the Borel Cantelli lemma to argue as follows.  First, we identify the class of tournaments on ${\mathbb Z}^+$ with the unit interval $[0,1]$ endowed with Lebesgue measure $\lambda$.  Then the Borel Cantelli lemma implies that the Lebesgue measure of those tournaments for which $\vert S_{n^2}-\e(S_{n^2})\vert>n^{3/2}\log n$ occurs infinitely often is zero, or, equivalently that for each tournament $T$ on ${\mathbb Z}^+$ outside of an exceptional set of measure zero, there exists $N=N(T)$ such that 
\[n\ge N(T)\Rightarrow S_{n^2}\in[\e(S_{n^2})-n^{3/2}\log n, \e(S_{n^2})+n^{3/2}\log n].\]
The goal is to show that the maximal term between the ``checkpoints" determined by the subsequence $a_n=n^2$ cannot be too badly behaved.  For any $N\in{\mathbb Z}^+$, denote by $T_N$ the tournament induced on $N$ by $T$.  For $1\le i\le n^2$, let $I_i$ and $O_i$ be respectively the in- and out-degrees of vertices in $T_{n^2}$, and let $I_i'$ and $O_i'$ be the in- and out-degrees of vertices in $\{1,2,\ldots,n^2\}$ to the ``new" vertices $\{n^2+1,\ldots,j\}$, where $n^2+1\le j\le (n+1)^2-1$.  By Azuma's inequality,
\begin{eqnarray}\p\lr\bigcup_{i=1}^{n^2}\vert I_i-O_i\vert>\lambda\rr&\le&n^2\p(\vert I_1-O_1\vert>\lambda)\nonumber\\
&\le&2n^2\exp\{-\lambda^2/8n^2\},\nonumber\end{eqnarray}
so that
\[\p(A_n^C):=\p\lr\bigcup_{i=1}^{n^2}\vert I_i-O_i\vert>n{\sqrt{40\log n}}\rr\le\frac{2}{n^3}.\]
A similar analysis yields
\[\p(B_n^C):=\p\lr\bigcup_{i=1}^{n^2}\vert I_i'-O_i'\vert>{\sqrt{80n\log n}}\rr\le\frac{2}{n^3}.\] Finally $\p(C_n^C):=\p(diam(T_{n^2})\ge 3)$ is exponentially small.  Thus, for $j=n^2,\ldots, (n+1)^2-1$,
\begin{eqnarray}
&&\p\lr\vert S_j-\e(S_j)\vert>j^{3/4}\log j; \vert S_{n^2}-\e(S_{n^2})\vert\le n^{3/2}\log n\rr\nonumber\\
{}&\le&\p\lr\vert S_j-\e(S_j)\vert>2n^{3/2}\log n; \vert S_{n^2}-\e(S_{n^2})\vert\le n^{3/2}\log n\rr\nonumber\\
&\le&\p(D_j)+\frac{5}{n^3},
\end{eqnarray}
where
\[D_j=\{\vert S_j-\e(S_j)\vert>2n^{3/2}\log n; \vert S_{n^2}-\e(S_{n^2})\vert\le n^{3/2}\log n; A_n,B_n,C_n\}.\]
Now if $diam(T_N)=2$, it follows that a vertex $j$ is Seymour iff $O_j\le I_j$.  Thus, if originally the number of Seymour vertices is within $n^{3/2}\log n$ of $\e(S_{n^2})$, and if $\vert S_j-\e(S_j)\vert>2n^{3/2}\log n$, then what may have caused this?  Note that $\vert \e(S_j)-\e(S_{n^2})\vert\le Kn$ for some constant $K$ and for each $j=n^2+1,\ldots,(n+1)^2$.  Also, we may assume (as a worst case scenario) that the linear number of ``new" vertices cause a change to the Seymour status of $T_j$ by an amount equal to their magnitude.  This means that for large $n$, at least $(n^{3/2}\log n)/2$ of the original $n^2$ vertices must have ``switched" their Seymour status.  Now since $\vert I_j-O_j\vert\le n{\sqrt{40\log n}}$ and $\vert I_j'-O_j'\vert\le {\sqrt{80n\log n}}$, no $j$ with
\[{\sqrt{80n\log n}}\le\vert I_j-O_j\vert\le n{\sqrt{40\log n}}\]
can switch.  Now for some $L>0$,
\[\p(\vert I_i-O_i\vert=r)\le \frac{L}{n}\] for each $r$ in $[0,{\sqrt{80n\log n}}]$, so that the expected number of $i$'s that switch is no more than $n^2\cdot L{\sqrt{80n\log n}}/n\le Mn^{3/2}\sqrt{\log n}$ for some $M$.  Moreover, since the numbers of these $i$'s that switch are independent, we have a high concentration of the number of vertices that switch around the expected value.  The probability that more than $(n^{3/2}\log n)/2$ switch is thus exponentially small.  It follows from (7) that $ \p(D_j)\le \frac{1}{n^3}$ and thus that
\[\p\lr\vert S_j-\e(S_j)\vert>j^{3/4}\log j; \vert S_{n^2}-\e(S_{n^2})\vert\le n^{3/2}\log n\rr\le\frac{6}{n^3},\]
so that
\[\sum_{j=n^2}^{(n+1)^2-1} \p\lr\vert S_j-\e(S_j)\vert>j^{3/4}\log j; \vert S_{n^2}-\e(S_{n^2})\vert\le n^{3/2}\log n\rr\le\frac{12}{n^2},\]
which yields
\[\sum_{n=1}^\infty \sum_{j=n^2}^{(n+1)^2-1}\p\lr\vert S_j-\e(S_j)\vert>j^{3/4}\log j; \vert S_{n^2}-\e(S_{n^2})\vert\le n^{3/2}\log n\rr<\infty,\]
so that, with ``i.o." representing ``infinitely often for $n=1,2,\ldots$ and $ j\in\{n^2,\ldots,(n+1)^2-1\}$"
\[\p\lr\vert S_j-\e(S_j)\vert>j^{3/4}\log j; \vert S_{n^2}-\e(S_{n^2})\vert\le n^{3/2}\log n\ \text{i.o.}\rr=0.\]
Since \[\p\lr\vert S_{n^2}-\e(S_{n^2})\vert>n^{3/2}\log n\ \text{infinitely often}\rr=0, \]
Theorem 2.4 follows, with $\epsilon=1/4$.  The case of general $\epsilon$ follows by taking larger and larger subsequences; in fact for arbitrary $\epsilon>0$ we start the subsequence $N=n^{1/2\epsilon}$ and an estimate on $\p(|S_N-\e(S_N)|>n^{0.5+(1/4\epsilon)}\log n)$.  There are $n^{(1-2\epsilon)/2\epsilon}$ ``new vertices," and the proof exploits the difference between $I, O$ and $I', O'$ as above.
\end{proof}

Notice that we never really proved an exponential inequality above; rather, we were able to show that the {\it conclusion} of the Borel Cantelli lemma held for deviations of the form $\p(\vert S_n-\e(S_n)\vert>n^{0.5+\epsilon})$.  However, we believe:
\begin{conj}  Theorem 2.4 can be improved to assert that for some $K$,
\[\sum_{n=1}^\infty\p\lr\vert S_n-\e(S_n)\vert>K{\sqrt {n\log n}}\rr<\infty.\] The proof might involve exponential rather than polynomial subsequences.
\end{conj}

\section{Random Digraphs}
In this section, we consider random digraphs $D(n,p)$ defined as follows: for each pair of vertices $(u,v) \in V(D)$, we place an arc from $u$ to $v$ with probability $p<1/2$; similarly, we place an arc from $v$ to $u$ with probability $p$. This construction gives no anti-edges, and the probability that there is no edge between $u$ and $v$ is $1-2p$. We allow for the case that $p=p_n\to0$ slowly as $n\to\infty$ or that $p=p_n=0.5-o(1)$.  In order to study the behavior of the number of Seymour vertices, we need the following concentration inequality from \cite{chung}.
\begin{lem}
Let $X$ be a sum of independent indicator random variables. Then for any $\epsilon>0$,
$$\mathbb{P}(X\geq (1+\epsilon)\mathbb{E}(X))\leq \left[\frac{e^{\epsilon}}{(1+\epsilon)^{1+\epsilon}}\right]^{\e(X)} .$$
\end{lem}
\begin{thm}
Let $D(n,p)$ be a random digraph on $n$ vertices with probability $\sqrt{\frac{(2+\epsilon)\log n}{n}} \leq p < \frac{1}{2}-\delta_n $, where $\epsilon>0$ is arbitrary and $\delta_n\to0$ is specified below.  Let $S$ be the set of its Seymour vertices.  Then
$\mathbb{E}(|S|) =n-o(1), n\to\infty.$
\end{thm}
\begin{proof}
Let $X=\vert S\vert$.
We have \begin{eqnarray}\e(X)&=& n \mathbb{P}( 1 \in S)\nonumber\\&\geq& n\mathbb{P}(1\in S; |N_1(1)| +|N_2(1)| = n-1)\nonumber\\
&=& n \mathbb{P}\lr|N_1(1)|\leq \frac{n-1}{2}; |N_t(1)|=0 \text{ for all }t\geq 3 \rr\nonumber\\ 
&\ge&n \mathbb{P}\lr|N_1(1)|\leq \frac{n-1}{2}\rr - n\mathbb{P}(|N_1(1)|+|N_2(1)|<n-1)\nonumber\\
&\ge&n \mathbb{P}\lr|N_1(1)|\leq \frac{n-1}{2 }\rr- n(n-1)(1-p)(1-p^2)^{n-2}\nonumber\\
&=&n \mathbb{P}\lr|N_1(1)|\leq \frac{n-1}{2}\rr- o(1),
\end{eqnarray}
provided that $p\ge {\sqrt{(2+\epsilon)\frac{\log n}{n}}}$.
But, by Lemma 3.1,
\begin{eqnarray}n \mathbb{P}\lr|N_1(1)|\leq \frac{n-1}{2}\rr&=&n\p\lr{\rm Bin}(n,p)\le\frac{n-1}{2}\rr\nonumber\\
&\ge&n\lr1-\p\lr{\rm Bin}(n,p)>\frac{n-1}{2}\rr\rr\nonumber\\
&\ge&n-n\lr\frac{2pe}{e^{2p}}\rr^{n/2}.\end{eqnarray}  Now the function $\varphi(p)=n\lr\frac{2pe}{e^{2p}}\rr^{n/2}$ tends to zero for each fixed $p\in(0,1/2)$, but on letting $p\to 1/2$ and setting $\lr\frac{2pe}{e^{2p}}\rr=1-\epsilon_n$, we see that the right side of (9) is of the form $n-ne^{-n\epsilon_n/2}=n-o(1)$ if $\epsilon_n=(2+\eta)\log n/n$, where $\eta>0$ is arbitrary.  Thus by (8) and (9) we have $\e(X)\ge n-o(1)$ if ${\sqrt{(2+\epsilon)\frac{\log n}{n}}}\le p\le 0.5-\delta_n$ for a $\delta_n$ that may be computed explicitly.  This proves the result.
\end{proof}
\begin{cor}
Let $D(n,p)$ be a random digraph with $p$ as in Theorem 3.2. As $n$ goes to infinity, $D$ has exactly $n$ Seymour vertices with high probability.
\end{cor}
\begin{proof}
Suppose $|S|\leq n-1$ with some probability $q>0$. Then
$\mathbb{E}(|S|)\leq (n-1)q+n(1-q)=n-q,$
which contradicts the fact that $\mathbb{E}(|S|)\geq n-o(1)$ as proven in Theorem 3.2.  Notice that our approach will also allow us to squeeze out results along the lines of an assertion that states that for an infinite tournament $T$, $\p(\vert S_n\vert\le n-1\ \text {infinitely often})=0$. \end{proof}


\section{Acknowledgments}
YZ was supported by Grant No. 14-12 from the Acheson J. Duncan Fund for the Advancement of Research in Statistics at The Johns Hopkins University.  ZC and LWH were supported by NSF Grant 0139286.  AG was supported by NSF Grants 0139286 and 1263009.

\end{document}